\documentclass[preprint,12pt]{elsarticle}

\usepackage{amsthm,amssymb,amsmath}
\usepackage{tikz}
\usetikzlibrary{shapes,matrix,arrows,calc,through}
\usetikzlibrary{positioning,fit,shapes.misc,matrix,decorations.text,shapes.geometric}

\newtheorem{theorem}{Theorem}

\newtheorem{corollary}{Corollary}

\newtheorem{definition}{Definition}

\journal{Discrete Applied Mathematics}

\begin{document}

\begin{frontmatter}



\title{Distance magic labelings of product graphs}


\author[1]{Rinovia Simanjuntak}
\ead{rino@math.itb.ac.id}
\author[2]{I Wayan Palton Anuwiksa}

\address[1]{Combinatorial Mathematics Research Group\\
Institut Teknologi Bandung, Bandung, Indonesia}
\address[2]{Master Program in Mathematics\\
Institut Teknologi Bandung, Bandung, Indonesia}

\begin{abstract}
A graph $G$ is said to be distance magic if there exists a bijection $f:V\rightarrow \{1,2, \ldots , v\}$ and a constant {\sf k} such that for any vertex $x$, $\sum_{y\in N(x)} f(y) ={\sf k}$, where $N_(x)$ is the set of all neighbours of $x$. 	

In this paper we shall study distance magic labelings of graphs obtained from four graph products: cartesian, strong, lexicographic, and cronecker. We shall utilise magic rectangle sets and magic column rectangles to construct the labelings.
\end{abstract}


\end{frontmatter}


\section{Introduction}
\label{Intro}

The notion of distance magic labeling was introduced separately in the PhD thesis of Vilfred \cite{Vi} in 1994 and an article by Miller {\it et. al} \cite{MRS} in 2003. A {\em distance magic labeling}, or $\Sigma$ labeling, of a graph $G$ is a bijection $\alpha:V(G) \rightarrow \{1,2, \ldots , v\}$ with the property that there is a constant {\sf k} such that at any vertex $x$, $\sum_{y\in N(x)} \alpha(y) ={\sf k}$, where $N(x)$ is the open neighborhood of $x$, i.e., the set of vertices adjacent to $x$. This labeling was introduced due to two different motivations; as a tool in utilizing magic squares into graphs and as a natural extension of edge-magic labeling.

It was proven in \cite{Vi} that every graph is a subgraph of a distance magic graph. A stronger result that every graph is an induced subgraph of a regular distance magic graph was then proved in \cite{ARSP}. A yet stronger result can also be found in \cite{RSP} where it is stated that every graph $H$ is an induced subgraph of a Eulerian distance magic graph $G$ where the chromatic number of $H$ is the same as $G$. All these results showed that there is no forbidden subgraph characterization for distance magic graph. Additionally, an application of the labeling in designing incomplete tournament is introduced in \cite{FKK}. Until now, the most important result for distance magic labeling is in \cite{OS,AKV} where it is shown that for a particular graph, the magic constant is unique and is determined by its fractional domination number. For more results, please refer to survey articles in \cite{AFK} and \cite{Ru}.

In 1999, Jinah \cite{Ji} introduced a variation of distance magic labeling. A \emph{closed distance magic labeling}  labeling of a graph $G$, is a bijection $\alpha:V(G) \rightarrow \{1,2, \ldots , v\}$ with the property that there is a constant {\sf k} such that at any vertex $x$, $\sum_{y\in N[x]} \alpha(y) ={\sf k}$, where $N[x]$ is the closed neighborhood of $x$, i.e., the set containing $x$ and all vertices adjacent to $x$.

Several previous results are essential in proving our results and we shall list them here. By $H_{n,p}$ we mean the complete multipartite graph with $p$ partite sets, each partite set consists of $n$ vertices. By $K_{n_1,n_2,\ldots,n_p}$ we mean the complete multipartite graphs of $p$ partite sets, where the first partite set consists of $n_1$ vertices, the second set consists of $n_2$ vertices, an so on.

\begin{theorem} \cite{MRS} \label{magicconstant}
Let $G$ be a $r$-regular graph with order $n$. If $G$ has distance magic labeling, then the magic constant is
\[k'=\frac{r(n+1)}{2}\]
\end{theorem}

\begin{theorem} \cite{Ji} \label{complement}
A graph $G$ is closed distance magic if and only if $\bar{G}$ is distance magic.
\end{theorem}

\begin{theorem} \cite{MRS} \label{Hnp}
For $n>1$ and $p>1$, $H_{n,p}$ is distance magic if only if $n$ is even or both $n$ and $p$ are odd.
\end{theorem}

\begin{theorem} \cite{Vi} \label{Kmn}
Let $m$ and $n$ be even positive integers, where $m \leq n$. $K_{m_1,\ldots,m_t,n_1,\ldots,n_t}$, where $m_i=m$ and $n_j=n$ for all $i$ and $j$, has distance magic labeling if and only if the following conditions hold:
\begin{itemize}
\item $m+n \equiv 0 \mod 4$
\item $1=2(2tn+1)^2-(2tm+2tn+1)^2$ or $m\geq (\sqrt{2}-1)n+\frac{\sqrt{2}-1}{2t}$
\end{itemize}
\end{theorem}

In this paper, we shall observe distance magic labeling for graphs obtained from four particular graph products: cartesian product, strong product, lexicographic product, and cronecker product. These results compliment and generalize results on distance magic labelings for some cartesian product graphs \cite{Ji,Ra}, some lexicographic product graphs \cite{MRS,SAS}, and joint product graphs \cite{SMA}.

\begin{definition}
Let $G$ and $H$ be two graphs.

The \emph{cartesian product}, $G\times H$, is a graph with vertex set $V(G) \times V(H)$. Vertices $(g,h)$ and $(g',h')$ are adjacent if and only if $g=g'$ and $h$ is adjacent to $h'$ in $H$, or $h=h'$ and $g$ is adjacent to $g'$ in $G$.

The \emph{strong product}, $G\boxtimes H$, is a graph with vertex set $V(G) \times V(H)$. Vertices $(g,h)$ and $(g',h')$ are adjacent if and only if $g=g'$ and $h$ is adjacent to $h'$ in $H$, or $h=h'$ and $g$ is adjacent to $g'$ in $G$, or $g$ is adjacent to $g'$ in $G$ and $h$ is adjacent to $h'$ in $H$.

The \emph{lexicographic product} of $G$ and $H$, $G\circ H$, is a graph with vertex set $V(G)\times V(H)$ that could be seen as a graph constructed by replacing every vertex of $G$ with a copy of $H$. Vertices $(g,h)$ and $(g',h')$ are adjacent if and only if either $g$ is adjacent to $g'$ in $G$ or $g=g'$ and $h$ is adjacent to $h'$ in $H$.

The \emph{cronecker product}, $G\otimes H$, is a graph with vertex set $V(G) \times V(H)$. Vertices $(g,h)$ and $(g',h')$
are adjacent if and only if $g$ is adjacent to $g'$ in $G$ and $h$ is adjacent to $h'$ in $H$.
\end{definition}

We shall frequently study products of a disjoint copy of complete multipartite graphs $mH_{n,p}$. Here we provide the notation for vertices of $mH_{n,p}$ along with their adjacencies. Suppose that $V(mH_{n,p})=\{v_1,v_2,\ldots,v_{mnp}\}$, the vertex sets of the components of $V(mH_{n,p})$ are $V_1,V_2,\ldots,V_m$ where
\[V_i=\{v_{(i-1)np+1},v_{(i-1)np+2},\ldots,v_{(i-1)np+np}\},\]
and the partitions of each $V_i$, $i=1,2,3,\ldots,m$, are $V_{i,j}$, $j=1,2,3,\ldots,p$, where
\[V_{i,j}=\{v_{(i-1)np+(j-1)n+1},v_{(i-1)np+(j-1)n+2},\ldots,v_{(i-1)np+(j-1)n+n}\}.\]
For each $i$, the subgraph induced by $V_i$ is the complete multipartite graph $H_{n,p}$. For each $i$ and $j$, the subgraph induced by $V_{i,j}$ is the null graph $nK_1$. For every $v,w\in V(mH_{n,p})$, $v$ is adjacent to $w$ if and only if $v\in V_{i,j}$ and $w\in V_{i,j'}$ where $j \neq j'$.	

Our main tool in constructing distance magic labelings for product graphs is the magic rectangle set that was introduced in \cite{Fr1,Fr2} and its generalisation, which we call the magic column rectangles. The definitions and properties of magic rectangle set and magic column rectangles are given in the next section.

\section{Magic Column Rectangles}

A \emph{magic rectangle set} $MRS(a,b;c)$ is a collection of $c$ arrays $(a \times b)$ whose entries are elements of  $\{1,2,\ldots,abc\}$, each appearing once, with all row sums in every rectangle equal to a constant $\rho$ and all column sums in every rectangle equal to a constant $\sigma$.

Froncek in \cite{Fr3} proved some conditions for the existence of magic rectangle sets of certain parameters as follow.

\begin{theorem}
If $a,b\equiv 0 \mod 2$ and $b\geq 4$, then there is magic rectangle set $MRS(a,b;c)$ for every $c\geq 1$.
\end{theorem}

\begin{theorem}
If $abc\equiv 1 \mod 2$ and $a,b>1$, then there is magic rectangle set $MRS(a,b;c)$ for every $c\geq 1$.
\end{theorem}

\begin{theorem}
If $a$ or $b$ is odd and $abc$ is even, then the magic rectangle set $ MRS(a,b;c) $ does not exist.
\end{theorem}

From the afore-mentioned theorems, we obtain the following characterisation for the existence of magic rectangle sets.
\begin{corollary} \label{MRS}
For $a>1$ and $b\geq 4$, a magic rectangle set $MRS(a,b;c)$ exists if and only if $a,b \equiv 0 \mod 2$ or $abc \equiv 1 \mod 2$.
\end{corollary}

Now were are ready to generalise magic rectangle sets as follow.
\begin{definition}
A \emph{magic column rectangle $MCR(m^{(1)},n^{(1)};p)$} is a matrix with the following properties:
\begin{itemize}
\item the order is $(m+n)\times p$, where $m \leq n$,
\item the entries are $1,2,3,\ldots,(m+n)p$, each appearing once,
\item \label{constant} it is partitioned into two blocks of submatrix, one of order $m\times p$ and the other of order $n\times p$, where the sum of entries in every column of each block is constant.
\end{itemize}
\end{definition}
It is clear that the magic constant in Definition \ref{constant} is the sum of all entries divided by the number of columns in those two blocks, that is $\frac{(mp+np)(mp+np+1)}{2}\cdot\frac{1}{2p}=\frac{(m+n)(mp+np+1)}{4}$.

We shall give necessary and sufficient conditions for the existence of an $MCR(m^{(1)},n^{(1)};p)$, where $m$ and $n$ are both even.
\begin{theorem} \label{MCRmnp}
Let $m$ and $n$ be two even positive integers, where $m \leq n$. A magic column rectangle $MCR(m^{(1)},n^{(1)};p)$ exists if and only if
\begin{itemize}
\item $m+n \equiv 0 \mod 4$ and
\item $1=2(2pn+1)^2-(2pm+2pn+1)^2$ or $m\geq (\sqrt{2}-1)n+\frac{(\sqrt{2}-1)}{2p}$
\end{itemize}
\end{theorem}
\begin{proof} Let $M$ be a matrix of order $(m+n)\times p$, where the entries are $1,2,\ldots,(m+n)p$. We partition $M$ into two blocks of order $m\times p$ and $n\times p$.

We shall use the entries of the block of order $m\times p$ as labels of vertices of $pK_{m}$ in such a way that the entries of one column correspond to vertices' labels of one component of $pK_m$. The entries of the block of order $n\times p$ is then used as labels of vertices of $pK_n$ such that the entries of one column of the block correspond with the labels of one component in $pK_n$.

Notice that $M$ is an $MCR(m^{(1)},n^{(1)};p)$ if and only if $G=pK_m \cup pK_n$ is closed distance magic. Clearly, $\bar{G} \approx K_{m_1,\ldots,m_p,n_1,\ldots,n_p}$, where $m_i=m$ and $n_i=n$. Then by utilising Theorems \ref{complement} and \ref{Kmn}, we complete the proof.
\end{proof}

\begin{definition} A \emph{magic column rectangle $MCR(n^{(p)};q)$} is a matrix with the following properties:
\begin{itemize}
\item the order is $np \times q$ where $n>1, p>1, q\geq1 $,
\item the entries are $1,2,3,\ldots,npq$, each appearing once,
\item it is partitioned into blocks of order $n\times1$, where the sum of entries in every block is constant. \label{constant2}
\end{itemize}
\end{definition}
It is clear that the magic constant in Definition \ref{constant2} is the sum of all entries divided by the number of blocks, that is  $\frac{npq(npq+1)}{2pq}=\frac{n(npq+1)}{2}$.

The necessary and sufficient conditions for the existence of an $MCR(n^{(p)};q)$ are given bellow.
\begin{theorem} \label{MCRnpq}
Let $n,p>1$ and $q\geq 1$. A magic column rectangle $MCR(n^{(p)},q)$ exists if and only if $n$ is even or all three of $n,p,q$ are odd.
\end{theorem}
\begin{proof}
Let $M$ be a matrix of order $np\times q$ where the entries are $1,2,\ldots,npq$, each appearing once. We partition $M$ into blocks of order $n\times 1$ and then use the entries of every block to label of vertices in one component of $pqK_n$.

Notice that $M$ is an $MCR(n^{(p)},q)$ if and only if $G=pqK_n$ is closed distance magic.
Since $\overline{G}\approx H_{n,pq}$, by Theorems \ref{complement} and \ref{Hnp}, the proof completes.
\end{proof}

In the following sections, the three matrices: magic rectangle set $MRS(a,b;c)$, magic column rectangles $MCR(m^{(1)},n^{(1)};p)$, and $MCR(n^{(p)};q)$ will be utilised to construct distance magic labelings for graphs obtained from Cartesian, strong, lexicographic, and kronecker products. We call such a labeling a matrix labeling which is defined as follow.
\begin{definition}
Let $G$ and $H$ be two graphs with $V(G)=\{g_1,g_2,\ldots,g_m\}$, $V(H)=\{h_1,h_2,\ldots,h_n\}$.
Suppose that $M$ is a matrix of order $m \times n$ whose entries are $1,2,\ldots,mn$, each appearing once.
If the operator $\ast$ is either $\boxtimes$, $\times$, $\otimes$, or $\circ$, we define an \textbf{$M$-matrix labeling} of $G*H$
as a map $\alpha:V(G) \times V(H) \rightarrow \{1,2,3,\ldots,mn\} $ where $\alpha(g_i,h_j)=M_{i,j}$.
\end{definition}

We shall start with distance magic labelings of Cartesian and strong products of a disjoint copy of complete multipartite graphs; both labelings are obtained from magic rectangle sets.

\section{Distance Magic Cartesian Product Graphs}

\begin{theorem}
Let $a,b\geq 1$, $m,p,q>1$, and $n\geq 4$. $aH_{m,p}\times bH_{n,q}$ is distance magic if and only if $m,n\equiv 0 \mod 2$ or $mnabpq \equiv 1 \mod 2$.
\end{theorem}
\begin{proof} Let $V(aH_{m,p})=\{v_1,v_2,...,v_{amp}\}$ and denote by $U(bH_{n,q})$, the set $V(bH_{n,q})=\{u_1,u_2,...,u_{bnd}\}$ where the adjacencies are given in Section \ref{Intro}. Notice that $aH_{m,p}\times bH_{n,q} \approx ab(H_{m,p} \times bH_{n,q})$. If $p,q \equiv 0 \mod 2$ or $mnabpq\equiv 1 \mod 2$, by Corollary \ref{MRS}, there exists a magic rectangle set $MRS(m,n;abpq)$.

From $MRS(m,n;abpq)$, we could obtain $abpq$ matrices of order $m \times n$ where the sum of entries in every row of each matrix is $\rho$ and the sum of entries in every column of each matrix is $\sigma$. We shall use these $abpq$ matrices to construct a matrix $M$ of order $amp \times bnq$, by stacking $ap$ matrices of order $m \times n$ vertically $bq$ times.

Let $\alpha$ be an $M-$matrix labeling for $H_{m,p} \times bH_{n,q}$. It will be shown that $\alpha$ is a distance magic labeling. For every $(v_i,u_j)\in V(aH_{m,p}\times bH_{n,q})) $, we obtain
\[w((v_i,u_j ))=\sum_{u_{j'}\sim u_j} \alpha((v_i,u_{j'} )) +\sum_{v_{i'}\sim v_i} \alpha((v_{i'},u_j )).\]

For every $v_i$ and $u_j$ in $V_{e,f}$ and $U_{g,h}$, respectively,
\[\sum_{u_{j'}\sim u_j} \alpha((v_i,u_{j'} )) =\sum_{u_{j'}\in U_{g,h'},h'\neq h} \alpha((v_i,u_{j'} )) = \sum_{u_{j'}\in U_{g}} M_{i,j'}-\sum_{u_{j'}\in U_{g,h}} M_{i,j'} =q\rho-\rho,\]
and
\[\sum_{v_{i'}\sim v_i} \alpha((v_{i'},u_j )) = \sum_{v_{i'}\in U_{e,f'},f'\neq f} \alpha((v_{i'},u_{j} ))= \sum_{u_{i'}\in V_{e}} M_{i',j}-\sum_{u_{i'}\in V_{e,f}} M_{i',j}=p\sigma-\sigma.\]
Thus $w((v_i,u_j ))=q\rho-\rho+p\sigma-\sigma $ is a constant for every $(v_i,u_j) \in V(aH_{m,c}\times bH_{n,d})$.

For the sufficiency, if $aH_{m,p} \times bH_{n,q}$ is distance magic, then there exists a distance magic labeling $\alpha$, which is an $M$-matrix labeling, where $M$ can be partitioned into $ab$ blocks of order $mp \times nq$ in such a way that each block corresponds with one component of $aH_{m,p} \times bH_{n,q}$.

Consider the block $B$ where $B_{i,j}=l((v_i,u_j))$ for $1\leq i\leq mp$ and $1\leq j\leq nq$. Notice that
\[w((v_1,u_1))=w((v_1,u_2))=\ldots=w((v_1,u_n)),\]
\[C+\sum_{v_i\in V_{1,j},j\neq 1}B_{i,1}=C+\sum_{v_i\in V_{1,j},j\neq 1}B_{i,2}=\ldots=C+\sum_{v_i\in V_{1,j},j\neq 1}B_{i,n},\]
where $C=\sum_{(u,v)\in K} l((u,v))$ and $K=N((v_1,u_1))\cap N((v_1,u_2))\cap \ldots N((v_1,u_n))$.

Thus,
\begin{equation}
\sum_{v_i\in V_{1,j},j\neq 1}B_{i,1}=\sum_{v_i\in V_{1,j},j\neq 1}B_{i,2}=\ldots=\sum_{v_i\in V_{1,j},j\neq 1}B_{i,n}.
\end{equation}
Similarly,
\begin{equation}
\sum_{v_i\in V_{1,j},j\neq 2}B_{i,1}=\sum_{v_i\in V_{1,j},j\neq 2}B_{i,2}=\ldots=\sum_{v_i\in V_{1,j},j\neq 2}B_{i,n},
\end{equation}
\[\vdots\]
\begin{equation} \tag{p}
\sum_{v_i\in V_{1,j},j\neq p}B_{i,1}=\sum_{v_i\in V_{1,j},j\neq p}B_{i,2}=\ldots=\sum_{v_i\in V_{1,j},j\neq p}B_{i,n}.
\end{equation}

By adding equations (1) to(p), we obtain
\[(p-1)\sum_{i\in V_1} B_{i,1}=(p-1)\sum_{i\in V_1} B_{i,2}=\ldots=(p-1)\sum_{i\in V_1} B_{i,n},\]
or
\begin{equation} \tag{*}
\sum_{i\in V_1} B_{i,1}=\sum_{i\in V_1} B_{i,2}=\ldots=\sum_{i\in V_1} B_{i,n}.
\end{equation}

By substracting equation (*) with each of the equations from (1) to (p), we obtain the following.
\[\sum_{v_i\in V_{1,j}}B_{i,1} = \sum_{v_i\in V_{1,j}}B_{i,2} = \ldots = \sum_{v_i\in V_{1,j}}B_{i,n}, j=1,\ldots,p.\]
Consequently, $\sum_{v_i\in V_{1,j}}B_{i,k}$ is a constant, say $\rho_1$, for $1\leq j\leq p, 1\leq k\leq n$.
Similarly, for $1\leq j\leq p, (l-1)n+1\leq k\leq ln, l=2,\ldots, q$, $\sum_{v_i\in V_{1,j}}B_{i,k}$ is a constant, say $\rho_l$.

Since cartesian product is commutative, similar arguments also hold for $B^t$, and so for $1\leq j\leq q, (l-1)m+1\leq k\leq lm, l=1,\ldots,p$, $\sum_{u_i\in U_{1,j}}B_{k,i}$ is a constant, say $\sigma_l$. 	

Now, we shall prove that $\rho_1=\rho_2=\ldots=\rho_d$. Notice that $w((v_1,u_1))=w((v_1,u_{(j-1)n+1}))$, or 
$(p-1)\rho_1+(q-1)\sigma_1=(p-1)\rho_j+(q-1)\sigma_1$, and so
$(p-1)(\rho_1-\rho_j)=0$. Since $p>1$, then $\rho_1=\rho_j$; and so, $\rho_1=\rho_2=\ldots=\rho_d$. 

Similarly, by considering $w((v_1,u_1))=w((v_{(j-1)m+1},u_1))$, we obtain $\sigma_1=\sigma_2=\ldots=\sigma_c$. Thus the matrix $B$ can be partitioned into $pq$ blocks of order $m \times n$, such that the row sum is $\sigma_1$ and the column sum is $\rho_1$ in every block. Considering the other $ab-1$ blocks of order $mp \times nq$, we could obtain similar result. Therefore $M$ can be partitioned into $abpq$ matrices of order $m \times n$, such that every row sum is $\sigma_1$ and every column sum is $\rho_1$ in every block; which mean $MRS(m,n;abpq)$ exists. By Corollary \ref{MRS}, we obtain $m,n \equiv 0 \mod 2$ or $ mnabpq \equiv 1 \mod 2$.
\end{proof}

\section{Distance Magic Strong Product Graphs}

\begin{theorem}
Let $a,b \geq 1$, $m,p,q>1$, and $n\geq 4$. $aH_{m,p} \boxtimes bH_{n,q} $ is distance magic if and only if $m,n \equiv 0 \mod 2$ or $mnabpq \equiv 1 \mod 2$.
\end{theorem}
\begin{proof}
Notice that $aH_{m,p} \boxtimes bH_{n,q} \approx ab(H_{m,p} \boxtimes bH_{n,q})$. Let $V(aH_{m,p})=\{v_1,v_2,\ldots,v_{amp}\}$ and we denote by $U(bH_{n,q})$, the set $V(bH_{n,q})=\{u_1,u_2,\ldots,u_{bnq}\}$ where the adjacencies are given in Section \ref{Intro}. If $m,n \equiv 0 \mod 2$ or $mnabpq \equiv 1 \mod 2$, then by Corrolary \ref{MRS}, there exists a magic rectangle set $MRS(m,n;abcd)$.

From $MRS(m,n;abcd)$, we substract $abpq$ matrices, each of order $m \times n$. Let the sum of entries in every row of each matrix is $\rho$ and the sum of entries in every column of each matrix is $\sigma$. We utilise those $abpq$ matrices to build a matrix $M$ of order $amc\times bnd$ by stacking $ap$ matrices of order $m \times n$ vertically, $bq$ times.

Let $\alpha$ be an $M-$matrix labeling on graph $aH_{m,p} \boxtimes bH_{n,q}$. We will show that $\alpha$ is a distance magic labeling. For every $(v_i,u_j)\in V(aH_{m,p} \boxtimes bH_{n,q}))$, we obtain
\[w((v_i,u_j))=\sum_{u_{j'}\sim u_j} \alpha((v_i,u_{j'} )) +\sum_{v_{i'}\sim v_i} \alpha((v_{i'},u_j )) + \sum_{u_{j'}\sim u_j,v_{i'}\sim v_i} \alpha((v_{i'},u_{j'}.))\]

Let $v_i\in V_{e,f}$ and $u_j\in U_{g,h}$. Then
\[\sum_{u_{j'}\sim u_j} \alpha((v_i,u_{j'} ))=\sum_{u_{j'}\in U_{g,h'},h'\neq h} \alpha((v_i,u_{j'} ))= \sum_{u_{j'}\in U_{g}} M_{i,j'}-\sum_{u_{j'}\in U_{g,h}} M_{i,j'}=q\rho-\rho,\]
\[\sum_{v_{i'}\sim v_i} \alpha((v_{i'},u_j ))=\sum_{v_{i'}\in V_{e,f'},f'\neq f} \alpha((v_{i'},u_{j} ))= \sum_{u_{i'}\in V_{e}} M_{i',j}-\sum_{u_{i'}\in V_{e,f}} M_{i',j}=p\sigma-\sigma,\]
and
\[\sum_{u_{j'}\sim u_j,v_{i'}\sim v_i} \alpha((v_{i'},u_{j'} ))=\sum_{{\scriptsize \begin{array}{c}
                                                                         u_{j'}\in U_{g,h'},h'\neq h, \\
                                                                         v_{i'}\in V_{e,f'},f'\neq f
                                                                      \end{array}}}
\alpha((v_{i'},u_{j'} )) = \sum_{{\scriptsize \begin{array}{c}
                                                                         u_{j'}\in U_{g,h'},h'\neq h, \\
                                                                         v_{i'}\in V_{e,f'},f'\neq f
                                                                      \end{array}}}
M_{i',j'}=pmq\rho.\]
Therefore $w((v_i,u_j))=q\rho-\rho+p\sigma-\sigma+cmd\rho $ is a constant for every $(v_i,u_j)\in V(aH_{m,p}\boxtimes bH_{n,q})$.

For the sufficiency, if $aH_{m,p} \boxtimes bH_{n,q}$ is distance magic, then there is distance magic labeling $\alpha$, which is an $M$-matrix labeling, where $M$ could be partitioned into $ab$ blocks of order $mp \times nq$ such that each block corresponds with one component of $aH_{m,p}\boxtimes bH_{n,q}$.

Consider a block $B$ where $B_{i,j}=l((v_i,u_j))$ for $1\leq i\leq mp$ and $1\leq j\leq nq$. Notice that
\[w((v_1,u_1))=w((v_1,u_2))=\ldots=w((v_1,u_n)),\]
\[C+\sum_{v_i\in V_{1,j},j\neq 1}B_{i,1}=C+\sum_{v_i\in V_{1,j},j\neq 1}B_{i,2}=\ldots=C+\sum_{v_i\in V_{1,j},j\neq 1}B_{i,n},\]
where $C=\sum_{(u,v)\in K} l((u,v))$ and $K=N((v_1,u_1))\cap N((v_1,u_2))\cap \ldots N((v_1,u_n))$, and so
\begin{equation} \tag{(1)}
\sum_{v_i\in V_{1,j},j\neq 1}B_{i,1}=\sum_{v_i\in V_{1,j},j\neq 1}B_{i,2}=\ldots=\sum_{v_i\in V_{1,j},j\neq 1}B_{i,n}.
\end{equation}

Similarly, we could obtain
\begin{equation} \tag{2}
\sum_{v_i\in V_{1,j},j\neq 2}B_{i,1}=\sum_{v_i\in V_{1,j},j\neq 2}B_{i,2}=\ldots=\sum_{v_i\in V_{1,j},j\neq 2}B_{i,n},
\end{equation}
\[\vdots\]
\begin{equation} \tag{p}
\sum_{v_i\in V_{1,j},j\neq c}B_{i,1}=\sum_{v_i\in V_{1,j},j\neq c}B_{i,2}=\ldots=\sum_{v_i\in V_{1,j},j\neq c}B_{i,n}.
\end{equation}

By adding equations (1) to (p), we obtain
\[(p-1)\sum_{i\in V_1} B_{i,1}=(p-1)\sum_{i\in V_1} B_{i,2}=\ldots=(p-1)\sum_{i\in V_1} B_{i,n},\]
\begin{equation} \tag{**}
\sum_{i\in V_1} B_{i,1}=\sum_{i\in V_1} B_{i,2}=\ldots=\sum_{i\in V_1} B_{i,n}.
\end{equation}

Now, substracting equation (**) with each of the equations from (1) to (p), we obtain
\[\sum_{v_i\in V_{1,j}}B_{i,1}=\sum_{v_i\in V_{1,j}}B_{i,2}=\ldots=\sum_{v_i\in V_{1,j}}B_{i,n}, 1 \leq j \leq p.\]
Consequently, $\sum_{v_i\in V_{1,j}}B_{i,k}$ is a constant, say $\rho_1$, for $1\leq j\leq p, 1\leq k\leq n.$ Similarly, we could obtain $\sum_{v_i\in V_{1,j}}B_{i,k}$ is also a constant, say $\rho_l$, for $1\leq j\leq p, (l-1)n+1 \leq k \leq ln, l=2,\ldots,q$.


Since strong product is commutative, the afore-mentioned reasoning also applies for $B^t$ and so
$\sum_{u_i\in U_{1,j}}B_{k,i}=y_1$ is a constant, say $\sigma_l$, for for $1\leq j\leq q, (l-1)m+1 \leq k \leq lm, l=1,\ldots,p$,

Now, we shall prove that $\rho_1=\rho_2=\ldots=\rho_q$. Notice that $w((v_1,u_1))=w((v_1,u_{(j-1)n+1}))$, and so
\[(p-1)\rho_1+(q-1)\sigma_1+\sum_{i=2}^{i=q}n\rho_i=(p-1)\rho_j+(q-1)\sigma_1+\sum_{i=1,i\neq j}^{i=q}n\rho_i,\]
\[(p-1)\rho_1+n\rho_j=(c-1)\rho_j+n\rho_1,\]
\[(p-1-n)(\rho_1-\rho_j)=0.\]
Since $p<n+1$, then $\rho_1=\rho_j$; and so, $\rho_1=\rho_2=\ldots=\rho_q$.

Considering $w((v_1,u_1))=w((v_{(j-1)m+1},u_1))$, we could also obtain $\sigma_1=\sigma_2=\ldots=\sigma_p$. Thus $B$ can be partitioned into $pq$ blocks of order $m\times n$, so that the row sum is $\sigma_1$ and the column sum is $\rho_1$ in every block. Considering other $ab-1$ blocks of order $mp \times nq$, we will get similar result. Thus, $M$ can be partitioned into $abpq$ matrices of order $m \times n$, such that the row sum is $\sigma_1$ and the column sum is $\rho_1$ in every block; and so, $MRS(m,n;abpq)$ exists. By Corollary \ref{MRS}, we obtain $m , n \equiv 0 \mod 2$ or $mnabpq \equiv 1 \mod 2$.
\end{proof}

\section{Distance Magic Lexicographic Product Graphs}

In the previous sections, we constructed distance magic labelings for cartesian and strong products of two disjoint copies of complete multipartite graphs. Construction of labelings for cartesian and strong products of more general graphs is still unknown. However, for lexicographic product, we have obtained sufficient conditions for the existence of distance magic labelings for product of any regular graph with a disjoint copy of complete multipartite graphs or with a null graph. We shall utilise a magic column rectangle $MCR(n^{(p)};q)$ in the construction of the labelings.

\begin{theorem} \label{lexicographic}
Let $m \geq 1$ and $n,p>1$. For an $r$-regular graph $G$ of order $m$, if $n \equiv 0 \mod 2$ or $apm \equiv 1 \mod 2$, then  $G\circ aH_{n,p}$ is distance magic.
\end{theorem}
\begin{proof}
Let $V(G)=\{v_1,v_2,\ldots,v_m\}$ and denote by $U(aH_{n,p})$, the set $V(aH_{n,p})=\{u_1,u_2,\ldots,u_{anp}\}$, where the adjacencies are given in Section \ref{Intro}. If $n \equiv 0 \mod 2$ or $apm \equiv 1 \mod 2$, then by Theorem \ref{MCRnpq}, there is a magic column rectangle $MCR(n^{(ap)};m)$, say $N$, and we denote by $M$, the transpose of $N$. We then could partition $M$ into blocks of order $1 \times n$, so that the sum of entries in every block is constant, say $\rho$.

Let $\alpha$ be an $M$-matrix labeling of $G\circ aH_{n,p}$. We shall prove that $\alpha$ is a distance magic labeling. For every $(v_i,u_j)\in V(G\circ aH_{n,p})$, we obtain
\[w((v_i,u_j))=\sum_{v_{i'}\sim v_i,u\in V(G)} \alpha((v_{i'},u )) + \sum_{u_{j'}\sim u_j} \alpha((v_i,u_{j'} ))=rap\rho+a(p-1)\rho.\]
Therefore $G\circ aH_{n,p}$ has a distance magic labeling.
\end{proof}

The necessity of the sufficient conditions in Theorem \ref{lexicographic} is still unknown for general regular graphs. However in the case of a disjoint copy of complete multipartite graphs, the sufficient conditions are indeed necesarry for a distance magic labeling to exist.
\begin{theorem}
Let $m,n,c,d >1$. $aH_{m,p} \circ bH_{n,q}$ is distance magic if and only if $n \equiv 0 \mod 2$ or $mnabpq \equiv 1 \mod 2$.
\end{theorem}
\begin{proof}
If $n \equiv 0 \mod 2$ or $mnabpq \equiv 1 \mod 2$, then, by Theorem \ref{lexicographic}, $aH_{m,p} \circ bH_{n,q}$ has a distance magic labeling.

Notice that $aH_{m,p} \circ bH_{n,q} \approx ab(H_{m,p}\circ bH_{n,q})$. Let $V(aH_{m,p} )=\{v_1,v_2,\ldots,v_{amp}\}$ and denote by $U(bH_{n,q})$, the set $V(bH_{n,q})=\{u_1,u_2,\ldots,u_{bnq}\}$ where the adjacencies are given in Section \ref{Intro}.

If $aH_{m,p} \circ bH_{n,q}$ is distance magic then there exists a distance magic labeling $\alpha$, which is an $M$-matrix labeling for a matrix $M$. Partition $M$ into $ab$ blocks of order $mp \times nq$ so that each block corresponds with one component of $aH_{m,p}\circ bH_{n,q}$.

Consider the block $B$ where $B_{i,j}=\alpha((v_i,u_j))$ for $1\leq i\leq mp$ and $1\leq j\leq nq$. Notice that
\[w((v_1,u_1))=w((v_2,u_1))=\ldots=w((v_m,u_1))\]
\[C+\sum_{u_i\in U_{1,j},j\neq 1}B_{1,i}=C+\sum_{u_i\in U_{1,j},j\neq 1}B_{2,i}=\ldots=C+\sum_{u_i\in U_{1,j},j\neq 1}B_{m,i}.\]
where $C=\sum_{(u,v)\in K} l((u,v))$ and $K=N((v_1,u_1))\cap N((v_2,u_1))\cap \ldots N((v_m,u_1))$, and so
\begin{equation} \tag{1}
\sum_{u_i\in U_{1,j},j\neq 1}B_{1,i}=\sum_{u_i\in U_{1,j},j\neq 1}B_{2,i}=\ldots=\sum_{u_i\in U_{1,j},j\neq 1}B_{m,i}.
\end{equation}
Applying similar procedures to vertices $(v_1,u_i),(v_2,u_i),\ldots,(v_m,u_i)$ where $i=n+1,2n+1,\ldots,(q-1)n+1$, we obtain
\begin{equation} \tag{2}
\sum_{u_i\in U_{1,j},j\neq 2}B_{1,i}=\sum_{u_i\in U_{1,j},j\neq 2}B_{2,i}=\ldots=\sum_{u_i\in U_{1,j},j\neq 2}B_{m,i}.
\end{equation}
\[\vdots\]
\begin{equation} \tag{q}
\sum_{u_i\in U_{1,j},j\neq q}B_{1,i}=\sum_{u_i\in U_{1,j},j\neq q}B_{2,i}=\ldots=\sum_{u_i\in U_{1,j},j\neq q}B_{m,i}.
\end{equation}
By adding all equations (1) up to (q), we obtain
\[(q-1)\sum_{i\in U_1} B_{1,i}=(q-1)\sum_{i\in U_1} B_{2,i}=\ldots=(q-1)\sum_{i\in U_1} B_{m,i},\]
or
\begin{equation} \tag{*}
\sum_{i\in U_1} B_{1,i}=\sum_{i\in U_1} B_{2,i}=\ldots=\sum_{i\in U_1} B_{m,i}
\end{equation}
By substracting equation (*) with each of the equations from (1) to (q), we obtain
\[\sum_{u_i\in U_{1,j}}B_{1,i}=\sum_{u_i\in U_{1,j}}B_{2,i}=\ldots=\sum_{u_i\in U_{1,j}}B_{m,i},\]
for $j = 1,2,\ldots, q$.
Consequently, $\sum_{u_i\in U_{1,j}}B_{k,i}$ is constant, say $\rho_1$, for $1\leq j\leq q, 1\leq k\leq m$.

Similarly, we could obtain $\sum_{u_i\in U_{1,j}}B_{k,i}$ is constant, say $\rho_l$, for $1\leq j\leq q, (i-1)m+1\leq k\leq lm$.
Now, we shall show that $\rho_1=\rho_2=\ldots=\rho_p$. Notice that
\[w((v_1,u_1))=w((v_{(j-1)m+1},u_1))\]
\[(q-1)\rho_1+mq\sum_{i=2}^{p}\rho_i=(q-1)\rho_j+mq\sum_{i=1,i\neq j}^{p}\rho_i\]
\[(q-1)\rho_1+mq\rho_j=(q-1)\rho_j+mq\rho_1\]
\[(q-1-mq)(\rho_1-\rho_j)=0\]
Since $mq+1>q$, then $\rho_1=\rho_j$ and $\rho_1=\rho_2=\ldots=\rho_p$.

Therefore the matrix $B$ can be partitioned into $mpq$ blocks of order $1\times n$, so that the sum of every block is $\rho_1$. Treating the other $ab-1$ blocks of order $mp\times nq$ similarly, we shall be able to partition $M$ into $abmpq$ matrices of order $1\times n$, such that the sum of every block is $\rho_1$. Consequently, $M^t$ can be partitioned into $abmpq$ matrices of order $n\times 1$, such that the sum of every block is $\rho_1$, which means that a $MCR(n^{(bq)};amp)$ exists. By Theorem \ref{MCRnpq}, we obtain $n \equiv 0 \mod 2$ or $mnabpq \equiv 1 \mod 2$.
\end{proof}

\begin{theorem} \label{nK1}
Let $m \geq 1$, $n>1$, and $G$ be an $r$-regular graph of order $m$. If
\begin{enumerate}
\item $n \equiv 0 \mod 2$, or
\item $mn \equiv 1 \mod 2$, or
\item $G$ is distance magic,
\end{enumerate}
then $G\circ nK_1$ is also distance magic.
\end{theorem}
\begin{proof}
Let $V(G)=\{v_1,v_2,\ldots,v_m\}$ and denote by $U(nK_1)$, the set $V(nK_1)=\{u_1,u_2,\ldots,u_{n}\}$.

If $n \equiv 0 \mod 2$ or $mn \equiv 1 \mod 2$, then by Theorem \ref{MCRnpq}, there exists an $MCR(n^{(1)};m)$, say $N$. Let $M=N^t$, then we could partition $M$ into blocks of order $1 \times n$, so that the sum of entries in every block is constant, say $\rho$. Let $\alpha$ be an $M$-matrix labeling of $G\circ nK_1$. For every $(v_i,u_j)\in V(G\circ nK_1)$, 
\[w((v_i,u_j ))=\sum_{v_{i'}\sim v_i,u\in V(G)} \alpha((v_{i'},u )) + \sum_{u_{j'}\sim u_j} \alpha((v_i,u_{j'} ))=r\rho.\] 
Thus $ w((v_i,u_j ))=r\rho$ is a constant and this means that $\alpha$ is a distance magic labeling.

If $n \equiv 1 \mod 2$, then $n-1 \equiv 0 \mod 2$. By Theorem \ref{MCRnpq}, there exists an $MCR((n-1)^{(1)};m)$, say $O$, such that the column sum is a constant, say $\varrho$. If $G$ is distance magic, let $\beta$ be a distance magic labeling of $G$ with magic constant ${\sf k}$. Let $P$ be a matrix of order $m\times 1$ where $P_{i,1}=\beta(v_i)$. Let $M=[P \ O^t]$ and $\beta'$ be an $M$-matrix labeling of $G\circ nK_1 $, then $w(v_i,u_j)={\sf k}+r\varrho$ for all $i$ and $j$. Therefore $G\circ nK_1 $ is also distance magic.
\end{proof}

The following theorem shows that rearanging a matrix $M$ corresponding to an $M$-matrix distance magic labeling for a graph $G$ could provide us a labeling for a graph other than $G$ but with the same order.
\begin{theorem}
Let $a \equiv 0 \mod b$, $n \geq 1$, and $G$ be a regular graph. If $aG \circ nK_1$ is distance magic then $bG \circ \frac{a}{b}(nK_1)$ is also distance magic.
\end{theorem}
\begin{proof}
Let $\alpha$ be a distance magic labeling for $aG \circ nK_1$, with magic constant $\sf k$. $\alpha$ is an $M$-matrix labeling, for a matrix $M$. Let $M=[B_1^t \ B_2^t$ \ldots $B_a^t]^t$ where $B_1$,\ldots,$B_a$ are blocks of order $m \times n$. We utilise those blocks to create a matrix $N$ of order $bm \times \frac{a}{b}n$ such that if $N$ is partitioned into blocks of order $a\times n$, then the blocks are $B_1$,\ldots,$B_a$. If $\alpha'$ is an $N$-matrix labeling of $bG\circ \frac{a}{b}(nK_1)$, then $w(v)=\frac{a}{b}\sf k$ for all $v\in V(bG \circ \frac{a}{b}(nK_1))$. Therefore $bG \circ \frac{a}{b}(nK_1)$ is also distance magic.
\end{proof}

\section{Distance Magic Cronecker Product Graphs}

The last product to be dealt with is the cronecker product. We shall start with a relation between distance magic labelings for lexicographic product to those for cronecker product.

\begin{theorem} \label{lextocron}
Let $a \geq 1$, $n,p>1$, and $G$ be a regular graph. Then the following statements are equivalent:
\begin{description}
  \item[(i)] $apG \circ nK_1$ is distance magic
  \item[(ii)] $aH_{n,p} \otimes G$ is distance magic
\end{description}
\end{theorem}
\begin{proof}
Let $G$ be a $r$-regular graph of order $m$.

\textbf{(i)$\Rightarrow$(ii)} Let $\alpha$ be a distance magic labeling for $apG\circ nK_1$ with magic constant $\sf k$, which is an $M$-matrix labeling for a matrix $M$. We partition the matrix $M$ into $ap$ blocks of order $m \times n$, say $B_1,\ldots,B_{ap}$.

Now, consider the blocks $B_1^t,\ldots,B_{ap}^t$ and we utilise these blocks to create a matrix $N$ of order $anp \times m$ such that if the matrix $N$ is partitioned into blocks of order $n \times m$, then the blocks are $B_1^t,\ldots,B_{ap}^t$. Let $\alpha'$ be an $N$-matrix labeling for $aH_{n,p}\otimes G$. Then $w(v)=(p-1)r\sf k$ for all $v \in V(aH_{n,p}\otimes G)$. Therefore, $aH_{n,p}\otimes G$ is distance magic.

\textbf{(ii)$\Rightarrow$(i)} Let $V(aH_{n,p})=\{ v_1,v_2,\ldots,v_{anp}\} $ and denote by $U(G)$ the set $V(G)=\{u_1,u_2,\ldots,u_m\}$ where the adjacencies are given in Section \ref{Intro}. Notice that $aH_{n,p} \otimes G \approx a(H_{n,p} \otimes G)$.

Since $aH_{n,p}\otimes G$ is distance magic, then there exists a distance magic labeling $\alpha$ with magic constant $\sf k$, which is an $M$-matrix labeling for a matrix $M$. Partition $M$ into $a$ blocks of order $np\times m$ so that each block corresponds to one component of $aH_{n,p}\otimes G$.

Consider a block $B$ where $B_{i,j}=l((v_i,u_j))$ for $1\leq i\leq np$ and $1\leq j\leq m$. Notice that for $u\in U(G)$, $\sum_{i=1}^pw((v_{(i-1)n+1)},u))=(p-1)\sum_{u\sim u_j} \sum_{i=1}^{np} B_{i,j}$. Since $\sum_{i=1}^pw((v_{(i-1)n+1)},u))=\sum_{i=1}^p\sf k=p\sf k$, we get $p {\sf k}=(p-1)\sum_{u\sim u_j} \sum_{i=1}^{np} B_{i,j}$, or $\sum_{u\sim u_j} \sum_{i=1}^{np} B_{i,j}=p{\sf k}/(p-1)$.

Let $V_1,V_2,\ldots,V_p$ be the partition of $\{v_1,v_2,\ldots,v_{np}\}$ such that each $V_i$ induces a graph $nK_1$ for all $i$.
For $v\in V_k$, $w((v,u))=\sum_{u\sim u_j} \sum_{i=1}^{np} B_{i,j}-\sum_{u\sim u_j,v_i\in V_k} B_{i,j}$. Therefore
${\sf k}=\sum_{u\sim u_j} \sum_{i=1}^{np} B_{i,j}-\sum_{u\sim u_j,v_i\in V_k} B_{i,j}$ or 
$\sum_{u\sim u_j,v_i\in V_k} B_{i,j}=p{\sf k}/(p-1)-{\sf k}={\sf k}/(p-1)$.	

Now, let the other $(a-1)$ blocks be $B_2,B_3,\ldots,B_a$ and we partition each block into blocks of order $n\times m$. All together, we obtain $ap$ blocks of order $n\times m$, say $A_1,A_2,\ldots,A_{ap}$. Let $N=[A_1 \ A_2 \ A_3 \ldots A_{ap}]^t$ and $\alpha'$ be an $N$-matrix labeling for $apG\circ nK_1$. Then $w(v)={\sf k}/(p-1)$ for all $v\in V(apG\circ nK_1)$. Therefore, $apG\circ nK_1 $ is distance magic.
\end{proof}

Combining Theorems \ref{nK1} and \ref{lextocron}, we obtain
\begin{corollary} \label{tpG}
Let $n>1$, $tpm>1$, and $G$ be a regular graph of order $m$. If $n \equiv 0 \mod2$ or $ntpm \equiv 1 \mod2$ or $tpG$ is distance magic then $tH_{n,p} \otimes G$ is also distance magic.
\end{corollary}

In the following theorems, we shall consider distance magic labelings of graphs which are cronecker products of disjoint copies of complete multipartite graphs with other graphs.
\begin{theorem}
Let $m\geq 2$, $n>1$, $p \geq 1$, and $\prod_{i=1}^{m}a_i\equiv 1 \mod 2$. $a_1H_{n,p} \otimes a_2H_{n,p} \otimes \ldots \otimes a_mH_{n,p}$ is distance magic if and only if  $n \equiv 0 \mod 2$ or $n,p \equiv \mod 2$.
\end{theorem}
\begin{proof}
Let $G=a_1H_{n,p} \otimes a_2H_{n,p} \otimes \ldots \otimes a_mH_{n,p}$. If $G$ is distance magic then by Theorem \ref{magicconstant}, the magic constant is $k'=\frac{n^m(p-1)^m(n^mp^m\prod_{i=1}^{m}a_i+1)}{2}$, which enforces $n \equiv 0 \mod 2$ or $n,p \equiv 1 \mod 2$.

If $n \equiv 0 \mod 2$ or $n,p \equiv 1 \mod 2$ then by Corollary \ref{tpG}, $G$ has a distance magic labeling.
\end{proof}

\begin{theorem} \label{HnpKm}
Let $m\geq 1$, $n,p>1$ and $ab \equiv 1 \mod2$. $aH_{n,p} \otimes bK_m$ is distance magic if and only if  $n \equiv 0 \mod 2$ or $n,p,m \equiv 1 \mod 2$.
\end{theorem}
\begin{proof}
If $n \equiv 0 \mod 2$ or $n,p,m \equiv 1 \mod 2$, by Corollary \ref{tpG}, $aH_{n,p} \otimes bK_m$ is distance magic.

Let $V(aH_{n,p})=\{v_1,v_2,\ldots,v_{anp}\}$ and denote by $U(bK_m)$, the set $V(bK_m)=\{u_1,u_2,\ldots,u_m\}$ where the adjacencies are given in Section \ref{Intro}. Notice that $aH_{n,p} \otimes bK_m \approx ab(H_{n,p} \otimes K_m)$. If $aH_{n,p}\otimes bK_m$ is distance magic, then there is a distance magic labeling $\alpha$ which is an $M$-matrix labeling, for a matrix $M$. Partition $M$ into $ab$ blocks of order $np \times m$ so that each block corresponds with one component of $aH_{n,p}\otimes bK_m$.

Consider the block $B$ where $B_{i,j}=\alpha((v_i,u_j))$ for $1\leq i\leq np$ and $1\leq j\leq m$. Notice that
\[w((v_1,u_1))=w((v_1,u_k)), 2\leq k\leq m,\]
\[C+\sum_{v_i\in V_{1,j},j\neq 1}B_{i,1}=C+\sum_{v_i\in V_{1,j},j\neq 1}B_{i,k},\]
where $C=\sum_{(u,v)\in K} \alpha((u,v))$ and $K=N((v_1,u_1))\cap N((v_1,u_k))$.
Thus,
\begin{equation} \tag{1}
\sum_{v_i\in V_{1,j},j\neq 1}B_{i,1}=\sum_{v_i\in V_{1,j},j\neq 1}B_{i,k}
\end{equation}
Similarly,
\begin{equation} \tag{2}
\sum_{v_i\in V_{1,j},j\neq 2}B_{i,1}=\sum_{v_i\in V_{1,j},j\neq 2}B_{i,k}
\end{equation}
\[\vdots\]	
\begin{equation} \tag{p}
\sum_{v_i\in V_{1,j},j\neq p}B_{i,1}=\sum_{v_i\in V_{1,j},j\neq p}B_{i,k}
\end{equation}

By adding equations (1) to (p), we obtain $(p-1)\sum_{i\in V_1} B_{i,1}=(p-1)\sum_{i\in V_1} B_{i,k}$, or
\begin{equation} \tag{*}
\sum_{i\in V_1} B_{i,1}=\sum_{i\in V_1} B_{i,k}
\end{equation}
By substracting equation (*) with each equation from (1) up to (p), we obtain $p$ equations:
\[\sum_{v_i\in V_{1,j}}B_{i,1}=\sum_{v_i\in V_{1,j}}B_{i,k}, j=1, \ldots, p.\]
Consequently, $\sum_{v_i\in V_{1,j}}B_{i,k}$ is constant, say $\rho$ for $1 \leq j \leq p, 1 \leq k \leq m$.

Thus, the matrix $B$ can be partitioned into $pm$ blocks of order $np\times m$, so that the sum of entries in every block is $\rho$. Similar results also hold for the other $ab-1$ blocks of order $np\times m$. In this case, $M$ can be partitioned into $apbm$ matrices of order $n\times 1$, such that the sum of entries in every block is $\rho$. This means that, $MCR(n^{(ap)},bm)$ exists, and so by Theorem \ref{MCRnpq}, we obtain $n \equiv 0 \mod 2$ or $npm \equiv 1 \mod 2$.
\end{proof}

\begin{theorem}
Let $a,b,n,p \geq 1$ and $m\geq 3$.
$aH_{n,p}\otimes bC_m$ is distance magic if and only if $n\equiv 0 \mod 2$ or $anpbm\equiv 1 \mod 2$.
\end{theorem}
\begin{proof}
If $n\equiv 0 \mod 2$ or $anpbm\equiv 1 \mod 2$, then by Theorem \ref{tpG}, $aH_{n,p}\otimes bC_m$ is distance magic.

If $aH_{n,p}\otimes bC_m$ is distance magic, then by Theorem \ref{lextocron}, $apbC_m \circ nK_1$ is also distance magic. Let $V(apbC_m)=\{v_1,\ldots,v_{apbm}\}$ and $V(nK_1)=\{u_1,\ldots,u_n\}$. Let $\alpha$ be a distance magic labeling of $apbC_m \circ nK_1$, with magic constant $\sf k$. Suppose that $\alpha$ is an $M$-matrix labeling, for a matrix $M$. Let us partition $M$ into $apb$ blocks of order $m\times n$. Let $B$ be one of those blocks such that $B_{i,j}=\alpha(v_i,u_j)$ where $1\leq i\leq m$ and $1\leq j\leq n$ and $b_i=\sum_{j=1}^{n}B_{i,j}$. 

Now, consider arbitrary $u\in V(nK_1)$. For $m=3$, $w((v_1,u))=w((v_2,u))=w((v_3,u))$ or $b_2+b_3=b_1+b_3=b_1+b_2$, which leads to $b_1=b_2=b_3$.

For $m\geq 4$ and $ 2\leq i\leq m-3$, we obtain $w((v_i,u))=w((v_{i+2},u))$, and so $b_{(i-1)\mod m}=b_{(i+3)\mod m}$.	
On the other hand, $w((v_m,u))=w((v_2,u))$, which leads to $b_{m-1}=b_3$.
Additionally, $w((v_{m-1},u))=w((v_1,u))$, or
$b_{m-2}=b_2$ and $w((v_{m-2},u))=w((v_m,u))$, or
$b_{m-3}=b_1$.

Thus, when $m$ is odd, we obtain $b_1=b_2=\ldots=b_m$, and so ${\sf k}=w((v_2,u))=b_1+b_3=2b_1$ or $b_1={\sf k}/2$.
Meanwhile, when $m$ is even, $b_1=b_3=b_5=\ldots=b_{m-1}$ and $b_2=b_4=b_6=\ldots=b_{m}$.
Notice that ${\sf k}=w((v_2,u))=b_1+b_3=2b_1$ or $b_1={\sf k}/2$, and ${\sf k}=w((v_3,u))=b_2+b_4=2b_2$ or $b_2={\sf k}/2$.

We conclude that every row sum of the matrix $B$ is ${\sf k}/2$. Similarly, we could obtain that every row sum of all the other $apb-1$ blocks is also ${\sf k}/2$. This means that $M^t$ is an $MCR(n^{(1)},apbm)$, and by Theorem \ref{MCRnpq}, $n\equiv 0 \mod 2$ or $anpbm \equiv 1 \mod 2$.
\end{proof}

In the two afore-mentioned theorems, we study the existence of distance magic labelings for cronecker product involving regular complete multipartite graphs. In the following theorems, we shall consider complete bipartite graphs that might not be regular. 

\begin{theorem} \label{Kmn}
Let $p,t \geq 1$ and $m, n$ be two even integers, where $m \leq n$. Suppose that $G$ is a regular graph with order $p$.
If the conditions
\begin{description}
\item[(i)] $(m+n) \equiv 0 \mod 4$ and
\item[(ii)] $1=2(2ptn+1)^2-(2ptm+2ptn+1)^2$ or $m\geq (\sqrt{2}-1)n+\frac{\sqrt{2}-1}{2pt}$
\end{description}
hold then $tK_{m,n} \otimes G$ is distance magic.
\end{theorem}
\begin{proof}
For $t=1$, let $V(K_{m,n})=\{v_1,v_2,\ldots,v_{m+n}\}$ and $V_1,V_2$ are the two partite sets of $V(G)$. Let $V(G)={u_1,u_2,\ldots,u_p}$. If both conditions (i) and (ii) hold, then by Theorem \ref{MCRmnp}, there exists $MCR(m^{(1)},n^{(1)};p)$, say $M$. If $M$ is partitioned into two blocks of matrices: one block of order $m\times p$ and another of order $n\times p$, then sum of entries in every column of those blocks is constant, say $\sigma$.

Let $\alpha$ be an $M$-matrix labeling on $K_{m,n} \otimes G$. We shall prove that $\alpha$ is distance magic.
For every $(v_i,u_j) \in V(K_{m,n}\otimes G)$, $w((v_i,u_j))=\sum_{u_{j'} \sim u_j,v_{i'}\sim v_i} \alpha((v_{i},u_{j'}))$.
Let $v_i\in V_a, a=1,2$, then
\[w((v_i,u_j))=\sum_{u_{j'}\sim u_j,v_{i'}\sim v_i} M_{i',j'}=\sum_{u_{j'}\sim u_j,v_{i'}\notin V_a} M_{i',j'}=(m-1)r\sigma\]
Thus, $\alpha$ is a distance magic labeling.

For $t\geq 2$, notice that $tK_{m,n}\otimes G\approx t(K_{m,n}\otimes G)$. If both conditions (i) and (ii) hold, then by Theorem \ref{MCRmnp}, there exists $MCR(m^{(1)},n^{(1)};pt)$, say $N$. Partition $N$ into $t$ blocks of matrices of order $(m+n)\times p$, that is $N=[B_1 \ B_2 \ldots B_t]$. We shall construct a matrix $O$ of order $(m+n)t\times p$ as $O=[B_1^t \ B_2^t \ldots B_t^t]^t$. Let $\beta$ be an $O$-matrix labeling on $tK_{m,n}\otimes G$. Since $tK_{m,n}\otimes G\approx t(K_{m,n}\otimes G)$, we could prove that $\beta$ is a distance magic labeling in a similar way as in the case of $t=1$.
\end{proof}

In general, it is still unknown whether the sufficient conditions in Theorem \ref{Kmn} are also necesarry. However they are indeed necesarry for the case of $G$ is a complete graph.

\begin{theorem}
Let $a,b,p \geq 1$ and $m, n$ be two even integers, where $m \leq n$.
$aK_{m,n} \otimes bK_p$ is distance magic if and only if
\begin{description}
\item[(i)] 	$(m+n)\equiv 0(mod 4)$ and
\item[(ii)] $1=2(2abpn+1)^2-(2abpm+2abpn+1)^2$ or $m\geq (\sqrt{2}-1)n+\frac{\sqrt{2}-1}{2abp}$.
\end{description}
\end{theorem}
\begin{proof} If conditions (i) and (ii) hold, by Theorem \ref{Kmn}, $aK_{m,n}\otimes bK_p $ is distance magic.

Now, if $aK_{m,n}\otimes bK_p$ is distance magic, then there exists a distance magic labeling $\alpha$, which is an $M$-matrix labeling, for a matrix $M$. Partition $M$ into $ab$ blocks of order $(m+n)\times p$, that is $B_1,B_2,\ldots,B_{ab}$, so that each block corresponds with one component of $aK_{m,n} \otimes bK_p$.

By using similar method as in the proof of Theorem \ref{HnpKm}, we obtain that $B_i$, $1\leq i\leq ab$, can be partitioned into a block of order $m\times p$ and a block of order $n\times p$ such that column sums of those blocks are a constant $\rho$. Consequently, the matrix $N=[B_1 \ B_2 \ldots B_{ab}]$ can be partitioned into a block of order $m\times abp$ and a block of order $n\times abp$ such that column sums of those blocks are also $\rho$. This means that an $MCR(m^{(1)},n^{(1)};abp)$ exists and, by Theorem \ref{MCRmnp}, the conditions (i) and (ii) must hold.
\end{proof}

\section*{Acknowledgments}

The first author was supported by Program Riset KK A 2017, funded by Institut Teknologi Bandung.

\end{document}